\documentclass[10pt]{article}
\usepackage[dvips]{graphicx}
\usepackage{amssymb,color}
\usepackage[latin1]{inputenc}
\usepackage{epic}
\usepackage{pstricks}
\usepackage{pst-node}
\usepackage{pstricks-add}
\usepackage[latin1]{inputenc}
\usepackage{color}
\usepackage{lettrine}
\usepackage{calc,pifont}
\usepackage[noindentafter,calcwidth]{titlesec}
\usepackage{pstricks}
\usepackage{pst-plot}
\usepackage{pst-node}
\usepackage{amsmath}
\usepackage{amssymb}
\usepackage{amsthm}
\usepackage{subfigure}
\usepackage{lineno}

\def\BBox{\kern  -0.2cm\hbox{\vrule width 0.2cm height 0.2cm}}

\newtheorem{teo}{Theorem}[section]
\newtheorem{coro}[teo]{Corollary}
\newtheorem{lema}[teo]{Lemma}

\newtheorem{conjecture}[teo]{Conjecture}
\theoremstyle{definition}

\theoremstyle{remark}

\title{On strong Skolem starters for $\mathbb{Z}_{pq}$}
\author{Adrián Vázquez-Ávila\thanks{adrian.vazquez@unaq.edu.mx}\\
{\small Subdirección de Ingeniería y Posgrado}\\
{\small Universidad Aeronáutica en Querétaro}\\
}

\date{}

\begin{document}
%\linenumbers
\maketitle
%%%%%%%%%%%%%%%%%%%%%%%%%%%%%%%%%%%%%%%%%%%%%%
\begin{abstract}
In 1991, N. Shalaby conjectured that any additive group $\mathbb{Z}_n$, where $n\equiv1$ or 3 (mod 8) and $n \geq11$, admits a strong Skolem starter and constructed these starters of all admissible orders $11\leq n\leq57$. Shalaby and et al. [O. Ogandzhanyants, M. Kondratieva and N. Shalaby, \emph{Strong Skolem Starters}, J. Combin. Des. {\bf 27} (2018), no. 1, 5--21] was proved if $n=\Pi_{i=1}^{k}p_i^{\alpha_i}$, where $p_i$ is a prime number such that $ord(2)_{p_i}\equiv 2$ (mod 4) and $\alpha_i$ is a non-negative integer, for all $i=1,\ldots,k$, then $\mathbb{Z}_n$ admits a strong Skolem starter. On the other hand, the author [A. Vázquez-Ávila, \emph{A note on strong Skolem starters}, Discrete Math. Accepted] gives different families of strong Skolem starters for $\mathbb{Z}_p$ than Shalaby et al, where $p\equiv3$ (mod 8) is an odd prime. Recently, the author [A. Vázquez-Ávila, \emph{New families of strong Skolem starters}, Submitted] gives different families of strong Skolem starters of $\mathbb{Z}_{p^n}$ than Shalaby et al, where $p\equiv3$ (mod 8) and $n$ is an integer greater than 1.

In this paper, we gives some different families of strong Skolem starters of $\mathbb{Z}_{pq}$, where $p,q\equiv3$ (mod 8) are prime numbers such that $p<q$ and $(p-1)\nmid(q-1)$.  
\end{abstract}

%%%%%%%%%%%%%%%%%%%%%%%%%%%%%%%%%%%%%%%%%%%%%%
\textbf{Keywords.} Strong starters, Skolem starters.

%{\bf MSC 2000.} ~05C35.
%%%%%%%%%%%%%%%%%%%%%%%%%%%%%%%%%%%%%%%%%%%%%%%%%%%%%%
%%%%%%%%%%%%%%%%%%%%%%%%%%%%%%%%%%%%%%%%%%%%%%%%%%%%%%INTRODUCTION
\section{Introduction}
Let $G$ be a finite additive abelian group of odd order $n=2k+1$, and let $G^*=G\setminus\{0\}$ be the set of non-zero elements of $G$. A \emph{starter} for $G$ is a set $S=\{\{x_i,y_i\},i=1,\ldots,k\}$ such that $\left\{\{x_i\}\cup\{y_i\}:i=1\ldots,k\right\}=G^*$
and $\left\{\pm(x_i-y_i):i=1,\ldots,k\right\}=G^*$. Moreover, if all elements $\left\{x_i+y_i:i=1,\ldots,k\right\}\subseteq G^*$ are different, then $S$ is called \emph{strong starter} for $G$. To see some works related to strong starters the reader may consult \cite{Avila,MR0325419,dinitz1984,MR1010576,MR0392622,MR1044227,MR808085,MR0249314,MR0260604}.  

Strong starters were first introduced by Mullin and Stanton in \cite{MR0234587} in constructing Room squares. Starters and strong starters have been useful to construct many combinatorial designs such as Room cubes \cite{MR633117}, Howell designs \cite{MR728501,MR808085}, Kirkman triple systems \cite{MR808085,MR0314644}, Kirkman squares and cubes \cite{MR833796,MR793636},  and factorizations of complete graphs \cite{MR0364013,Bao,MR2206402,DinitzSequentially,MR1010576,MR623318,MR685627,Meszka,AvilaC4free}.

Let $n=2k+1$, and $1<2<\ldots<2k$ be the order of $\mathbb{Z}_n^*$. A starter for $\mathbb{Z}_n$ is \emph{Skolem} if it can be written as $S=\{\{x_i,y_i\}:i=1,\ldots,k\}$ such that $y_i>x_i$ and $y_i-x_i=i$ (mod n), for $i=1,\ldots,k$. In \cite{ShalabyThesis}, it was proved the Skolem starter for $\mathbb{Z}_n$ exits if and only if $n\equiv1,3$ (mod 8). A starter which is both Skolem and strong is called \emph{strong Skolem starter}.

Shalaby in \cite{ShalabyThesis} proposed the following:

\begin{conjecture}
If $n\equiv1,3$ (mod 8) and $n \geq11$, then $\mathbb{Z}_n$ admits a strong Skolem starter. 
\end{conjecture} 

In \cite{Shalaby}, it was proved if $n=\Pi_{i=1}^{k}p_i^{\alpha_i}$, where $p_i$ is a prime such that $ord(2)_{p_i}\equiv 2$ (mod 4) and $\alpha_i$ is a non-negative integer, for all $i=1,\ldots,k$, then $\mathbb{Z}_n$ admits a strong Skolem starter, where $ord(2)_{p_i}$ is the order of the element 2 in $\mathbb{Z}_{p_i}$. In \cite{AvilaSkolem}, it was given different families of strong Skolem starters of $\mathbb{Z}_p$, where $p\equiv3$ (mod 8) is an odd prime, using a different method than in \cite{Shalaby}. Recently in \cite{AvilaSkolem2}, it was it was given different families of strong Skolem starters of $\mathbb{Z}_{p^n}$, where $p\equiv3$ (mod 8) and $n$ is an integer greater than 1, than in \cite{Shalaby}.

This paper is organized as follows. In Section \ref{sec:quadratic}, we recall some basic properties about quadratic residues and we present the strong Skolem starters of $\mathbb{Z}_{p}$ given in \cite{AvilaSkolem}; this idea is used in the main result of this paper, Theorem \ref{thm:main}. Finally, in section \ref{sec:main}, we give the main result of this paper, and we present one example. The main theorem states the following:
\begin{teo}
Let $p$ and $q$ be odd prime numbers such that $p,q\equiv3$ (mod 8) such that $p<q$ and $(p-1)\nmid(q-1)$. Then
$\mathbb{Z}_{pq}$ admits a strong Skolem starter. 
\end{teo}
%%%%%%%%%%%%%%%%%%%%%%%%
%%%%%%%%%%%%%%%%%%%%%%%%
%%%%%%%%%%%%%%%%%%%%%%%%
\section{A family of strong Skolem starters for $\mathbb{Z}_p$}\label{sec:quadratic}
Let $p$ be an odd prime power. An element $x\in\mathbb{Z}_p^*$ is called a \emph{quadratic residue} if there exists an element $y\in\mathbb{Z}_p^{*}$ such that $y^2=x$. If there is no such $y$, then $x$ is called a \emph{non-quadratic residue.} The set of quadratic residues of $\mathbb{Z}_p^{*}$ is denoted by $QR(p)$ and the set of non-quadratic residues is denoted by $NQR(p)$. It is well known that $QR(p)$ is a (cyclic) subgroup of $\mathbb{Z}_p^{*}$ of cardinality $\frac{p-1}{2}$ (see for example \cite{MR2445243}); also, if either $x,y\in QR(p)$ or $x,y\in NQR(p)$, then $xy\in QR(p)$, and if $x\in QR(p)$ and $y\in NQR(p)$, then $xy\in NQR(p)$.

The following theorems are well known results on quadratic residues. 
For more details of this kind of results the reader may consult \cite{burton2007elementary,MR2445243}.

\begin{teo}\label{col:menosuno}
	Let $p$ be an odd prime power, then
	\begin{enumerate}
		\item $-1\in QR(p)$ if and only if $p\equiv1$ (mod $4$).
		\item $-1\in NQR(p)$ if and only if $p\equiv3$ (mod $4$).
	\end{enumerate}
\end{teo}

\begin{teo}\label{inverso}
Let $p$ be an  odd prime. If $p\equiv3$ (mod $4$), then 
	\begin{enumerate}
		\item $x\in QR(p)$ if and only if $-x\in NQR(p)$.
		\item $x\in NQR(p)$ if and only if $-x\in QR(p)$.
	\end{enumerate}
\end{teo}

In \cite{Avila}, it was proved the following (see also \cite{MR0249314}):

\begin{lema}\cite{Avila}
If $p\equiv3$ (mod 4) is an odd prime power with $p\neq3$ and $\alpha$ is a generator of $QR(p)$, then  the following set
\begin{eqnarray*}\label{strong_1}
	S_\beta=\left\{\{x,\beta x\}: x\in QR(p)\right\},
\end{eqnarray*} 
is a strong starter for $\mathbb{Z}_p$, for all $\beta\in NQR(p)\setminus\{-1\}$. 
\end{lema}

In \cite{AvilaSkolem}, it was proved the following:

\begin{teo}\cite{AvilaSkolem}\label{thm:SkolemAdrian}
Let $p\equiv3$ (mod 8) be an odd prime and $\alpha$ be a generator of $QR(p)$, then the following strong starter $$S_\beta=\left\{\{x,\beta x\}: x\in QR(p)\right\}$$ for $\mathbb{Z}_p$ is Skolem, if $\beta=2$ and $\beta=\frac{1}{2}$.
\end{teo}

For the main result, it is needed the strong Skolem starter given in the above theorem.
%%%%%%%%%%%%%%%%%%%%%%%%
%%%%%%%%%%%%%%%%%%%%%%%%
%%%%%%%%%%%%%%%%%%%%%%%%
\section{A family of strong Skolem starters for $\mathbb{Z}_{pq}$}\label{sec:main}

The following definitions and notations are obtained in \cite{Shalaby}. Let $G_n$ be the group of units of the ring $\mathbb{Z}_n$ (elements invertible with respect to multiplication). It is denoted by $\langle x\rangle_n$ the cyclic subgroup of $G_n$ generated by $x\in G_n$. Also, we will use the notation $aB=\{ab: b\in B\}$, where $a\in\mathbb{Z}$ and $B\subseteq\mathbb{Z}$. On the other hand, it is denoted by $ord(x)_n$ the order of the element $x\in G_n$; hence, $ord(x)_n=|\langle x\rangle_n|$. Whenever the group operation is irrelevant, it will consider $G_n$ and its cyclic multiplicative subgroups $\langle x\rangle_n$ in the set-theoretical sense and denote them by $\underline{G}_n$ and $\langle\underline{x}\rangle_n$, respectively. 

Let $p,q\equiv3$ (mod 8) be odd prime numbers such that $p<q$ and $(p-1)\nmid(q-1)$. We have $\underline{G}_{pq}=\{x\in\mathbb{Z}_{pq}^*: gcd(x,pq)=1\}$, with $|\underline{G}_{pq}|=(p-1)(q-1)$, see for example \cite{Cohen:1995}. Hence, $p\mathbb{Z}_q^*$, $q\mathbb{Z}_p^*$ and $\underline{G}_{pq}$ forms a partition of $\mathbb{Z}_{pq}^*$, since every element $x\in\mathbb{Z}_{pq}^*$ lies in one and only one of these sets. Moreover, it is not difficult to prove that, if $r\in\mathbb{Z}_q^*$ is a primitive root, then $|\langle r \rangle_{pq}|=lcm(p-1,q-1)=\frac{(p-1)(q-1)}{2}$, since $(p-1)\nmid(q-1)$.

\begin{lema}\label{lema:-1}
Let $p,q\equiv3$ (mod 8) be odd prime numbers. If $x\in\mathbb{Z}_p^*$ and $y\in\mathbb{Z}_q^*$ are primitive roots, then $-1\not\in\langle x^2 \rangle_{pq}$ and $-1\not\in\langle y^2 \rangle_{pq}$.
\end{lema}
\begin{proof}
Recall that $G_{m}$ is the group of units of $\mathbb{Z}_{m}$. It is well known that the map $\Psi:G_{pq}\to G_{p}\times G_{q}$ defined by $\Psi(k_{pq})=(k_p,k_q)$, is an isomorphism between $G_{pq}$ and $G_p\times G_q$. Since $p$ and $q$ are prime numbers then $G_p=\mathbb{Z}_p^*$ and $G_q=\mathbb{Z}_q^*$. Let $x\in\mathbb{Z}_p^*$ and  $y\in\mathbb{Z}_q^*$ be primitive roots, and suppose that $-1\in\langle x^2\rangle_{pq}$. Then there exists $j\in\{0,\ldots,\frac{pq-1}{2}\}$ such that $x^{2j}=-1$ (mod $pq$). Hence, we have $\Psi(-1)=\Psi(x^{2j}_{pq})=(x^{2j}_p,x^{2j}_q)\neq(-1_p,-1_q)$ (by Theorem \ref{col:menosuno}). The proof is analogous if we suppose that $-1\in\langle y^2\rangle_{pq}$.
\end{proof}

\begin{lema}\label{lema:2}
Let $p,q\equiv3$ (mod 8) be odd prime numbers. If $x\in\mathbb{Z}_p^*$ and $y\in\mathbb{Z}_q^*$ are primitive roots, then $2\not\in\langle x^2 \rangle_{pq}$ and $2\not\in\langle y^2 \rangle_{pq}$. 
\end{lema}
\begin{proof}
Let $\Psi:G_{pq}\to G_{p}\times G_q$ given by $\Psi(k_{pq})=(k_p,k_q)$ as above, and let $x\in\mathbb{Z}_p^*$ and $y\in\mathbb{Z}_q^*$ be primitive roots. Suppose that $2\in\langle x^2\rangle_{pq}$. Then there exists $j\in\{0,\ldots,\frac{pq-1}{2}\}$ such that $x^{2j}=2$ (mod $pq$). Hence, we have $\Psi(2)=\Psi(x^{2j}_{pq})=(x^{2j}_p,x^{2j}_q)\neq(2,2)$, since if $p,q\equiv3$ (mod 8) then $2\not\in \langle x^2\rangle_p$ and $2\not\in \langle y^2\rangle_q$ (see for example \cite{Ireland}). The proof is analogous if we suppose that $2\in\langle y^2\rangle_{pq}$. 
\end{proof}

\begin{teo}\label{thm:main}
Let $p<q$ be odd prime numbers such that $p,q\equiv3$ (mod 8)  and $(p-1)\nmid(q-1)$. Then
$\mathbb{Z}_{pq}$ admits a strong Skolem starter.
\end{teo}
\begin{proof}
Let $r_1\in\mathbb{Z}_p^*$ and $r_2\in\mathbb{Z}_q^*$ be primitive roots of $\mathbb{Z}_p^*$ and $\mathbb{Z}_q^*$, respectively. Hence $\alpha_1=r_1^2$ is a generator of $QR(p)$ and $\alpha_2=r_2^2$ is a generator of $QR(q)$. Since $p\mathbb{Z}_q^*$, $q\mathbb{Z}_p^*$ and $\underline{G}_{pq}$ forms a partition of $\mathbb{Z}_{pq}^*$, $2\not\in \langle x^2\rangle_p$ and $2\not\in \langle y^2\rangle_q$, and by Lemma \ref{lema:2}, we can define
\begin{eqnarray*}
pS_{q}&=&\left\{\{px,2px\}:x\in QR(q)\right\}\\
qS_{p}&=&\left\{\{qx,2qx\}:x\in QR(p)\right\}\\
S_{pq}&=&\left\{\{x,2x\}:x\in\langle\alpha_2\rangle_{pq}\right\}\cup\left\{\{\lambda x,2\lambda x\}:x\in\langle\alpha_2\rangle_{pq}\right\},
\end{eqnarray*}
where $\lambda\not\in\langle\alpha_2\rangle_{pq}\cup2\langle\alpha_2\rangle_{pq}$. It is easy to see that $\{\pm px:x\in QR(q)\}=p\mathbb{Z}_{q}^*$, $\{\pm qx:x\in QR(p)\}=q\mathbb{Z}_{p}^*$ and $\{\pm x:x\in\langle\alpha_2\rangle_{pq}\}\cup\{\pm \lambda x:x\in\langle\alpha_2\rangle_{pq}\}=\underline{G}_{pq}$. Hence, the set $S=pS_{q}\cup qS_{p}\cup S_{pq}$ is a starter. 

Let define 
\begin{eqnarray*}
qS_p^+&=&\{3px:x\in QR(p)\}\\
pS_q^+&=&\{3qx:x\in QR(q)\}\\
S_{pq}^+&=&\{3x:x\in\langle\alpha_2\rangle_{pq}\}\cup\{3\lambda x:x\in\langle\alpha_2\rangle_{pq}\}
\end{eqnarray*}
Since $l_1x+l_2y\neq0$, for all different $x,y\in QR(p)\cup QR(q)\cup\langle\alpha\rangle_{pq}$, with $l_1, l_2\in\{1,p,q,\lambda\}$. Then $|qS_p^+|=|QR(p)|$, $|pS_q^+|=|QR(q)|$ and $|S_{pq}^+|=|\langle r_2\rangle_{pq}|$. By Lemma \ref{lema:-1}, it is not difficult to see that $|qS_p^+\cup pS_q^+\cup S_{p}q^+|=|qS_p^+|+|pS_q^+|+|S_{pq}^+|=\frac{p-1}{2}+\frac{q-1}{2}+\frac{(p-1)(q-1)}{2}=\frac{pq-1}{2}$. Hence, the set $S$ is strong.

Finally, we give an analogous proof of Theorem 2.4 given in \cite{AvilaSkolem} (see case (i)) to prove $S$ is Skolem. Let $pq=2t+1$ and $1<2<\cdots<2t$ be the order of the non-zero elements of $\mathbb{Z}_{pq}^*$. Define $Q_{\frac{1}{2}}=\{1,2,\ldots,t\}$. To prove that $S$ is Skolem, it is sufficient to prove that, if $2lx>lx$ then $lx\in Q_\frac{1}{2}$, and if $lx>2lx$ then $-lx\in Q_\frac{1}{2}$, where $l\in\{1,p,q,\lambda\}$. Suppose that $lx\in Q_\frac{1}{2}$, for $l\in\{1,p,q,\lambda\}$, then $2lx>lx$, which implies that $2lx-lx=lx\in Q_\frac{1}{2}$. On the other hand, if $lx\not\in Q_\frac{1}{2}$, for $l\in\{1,p,q,\lambda\}$, then $-lx\in Q_{\frac{1}{2}}$ (by Theorem \ref{col:menosuno} and Lemma \ref{lema:-1}). Hence, $2(-lx)>-lx$, which implies that $-2lx+lx=-lx\in Q_{\frac{1}{2}}$. Hence, the set $S$ is Skolem. 

Therefore, $\mathbb{Z}_{pq}$ admits a strong Skolem starter. 
\end{proof}

\begin{coro}\label{coro:final}
Let $p<q$ be odd prime numbers such that $p,q\equiv3$ (mod 8) and $(p-1)\nmid(q-1)$, and let $r_1\in\mathbb{Z}_p^*$ and $r_2\in\mathbb{Z}_q^*$ be primitive roots of $\mathbb{Z}_p^*$ and $\mathbb{Z}_q^*$, respectively. If $pS_{q}^{-}=\left\{\{px,2^{-1}px\}:x\in QR(q)\right\}$, $S_{p}^{-}=\left\{\{qx,2^{-1}qx\}:x\in QR(p)\right\}$ and
$S_{pq}^{-}=\left\{\{x,2^{-1}x\}:x\in\langle r_2^2\rangle_{pq}\right\}\cup\left\{\{\lambda x,2^{-1}\lambda x\}:x\in\langle r_2^2\rangle_{pq}\right\}$,	then set $S^{-}=pS_{q}^{-}\cup qS_{p}^{-}\cup S_{pq}^{-}$ is a strong Skolem starter for $\mathbb{Z}_{pq}$.
\end{coro}
\begin{proof}
The proof is analogous of Theorem \ref{thm:main} and using the case (ii) of Theorem \ref{thm:SkolemAdrian} given in \cite{AvilaSkolem}.
\end{proof}

\subsection{Example}
Consider $\mathbb{Z}_{11\cdot 19}$. We have $\mathbb{Z}_{11\cdot 19}^*=19\underline{G}_{11}\cup 11\underline{G}_{19}\cup\underline{G}_{11\cdot19}$. In this case $r_1=2$ is a primitive root of $\mathbb{Z}_{11}^*$ and $r_2=2$ is a primitive root of $\mathbb{Z}_{19}^*$. We have $\alpha_1=r_1^2$, $\alpha_2=r_2^2$ and $\lambda=3$. Hence, the pairs from $19\underline{G}_{11}$, $11\underline{G}_{19}$ and $\underline{G}_{11^2}$ are:
\begin{eqnarray*}
19\underline{G}_{11}&:& \{19,38\}, \{76,152\}, \{95,190\}, \{171,133\}, \{57,114\}.\\
11\underline{G}_{19}&:& \{11,22\}, \{44,88\}, \{176,143\}, \{77,154\}, \{99,198\}, \{187,165\},\\
&&\{121,33\}, \{66,132\}, \{55,110\}.\\
\underline{G}_{11\cdot19}&:& \{1,2\}, \{4,8\}, \{16,32\}, \{64,128\}, \{47,94\}, \{188,167\}, \{125,41\},\\
&&\{82,164\}, \{119,29\}, \{58,116\}, \{23,46\}, \{92,184\}, \{159,109\}, \{9,18\},\\
&&\{36,72\}, \{144,79\}, \{158,107\}, \{5,10\}, \{20,40\}, \{80,160\}, \{111,13\},\\
&&\{26,52\}, \{104,208\}, \{207,205\}, \{201,193\}, \{177,145\}, \{81,162\}, \{115,21\}\\
&&\{42,84\}, \{168,127\}, \{45,90\}, \{180,151\}, \{93,186\}, \{163,117\}, \{25,50\}\\
&&\{100,200\}, \{191,173\}, \{137,65\}, \{130,51\}, \{102,204\}, \{199,189\}, \{169,129\}\\
&&\{49,98\}, \{196,183\}, \{157,105\}, \{3,6\}\\
&\cup& \{12,24\}, \{48,96\}, \{192,175\}, \{141,73\}, \{146,83\}, \{166,123\}, \{37,74\},\\
&& \{148,87\}, \{174,139\}, \{69,138\}, \{67,134\}, \{59,118\}, \{27,54\}, \{108,7\},\\ 
&&\{14,28\}, \{56,112\}, \{15,30\}, \{60,120\}, \{31,62\}, \{124,39\}, \{78,156\},\\ 
&&\{103,206\}, \{203,197\}, \{185,161\}, \{113,17\}, \{34,68\}, \{136,63\}, \{126,43\},\\ 
&&\{86,172\}, \{135,61\}, \{122,35\}, \{70,140\}, \{71,142\}, \{75,150\}, \{91,182\},\\ 
&&\{155,101\}, \{202,195\}, \{181,153\}, \{97,194\}, \{179,149\}, \{89,178\},\\
&&\{147,85\}, \{170,131\}, \{53,106\}.
\end{eqnarray*}
This strong Skolem starter of $\mathbb{Z}_{11\cdot19}$ is the same of the Example 4.12 given in \cite{Shalaby}. Now, by Corollary \ref{coro:final}, we have a different strong Skolem starter of $\mathbb{Z}_{11\cdot19}$, using the sme parameters:
\begin{eqnarray*}
	19\underline{G}_{11}&:& \{19,114\}, \{76,38\}, \{95,152\}, \{171,190\}, \{57,133\}.\\
	11\underline{G}_{19}&:& \{11,110\}, \{44,22\}, \{176,88\}, \{77,143\}, \{99,154\}, \{187,198\},\\
	&&\{121,165\}, \{66,33\}, \{55,132\}.\\
\end{eqnarray*}	
\begin{eqnarray*}	
	\underline{G}_{11\cdot19}&:& \{1,105\}, \{4,2\}, \{16,8\}, \{64,32\}, \{47,128\}, \{188,94\}, \{125,167\},\\
	&&\{82,41\}, \{119,164\}, \{58,29\}, \{23,116\}, \{92,46\}, \{159,84\}, \{9,109\},\\
	&&\{36,18\}, \{144,72\}, \{158,79\}, \{5,107\}, \{20,10\}, \{80,40\}, \{111,160\},\\
	&&\{26,13\}, \{104,52\}, \{207,208\}, \{201,205\}, \{177,193\}, \{81,145\}, \{115,162\}\\
	&&\{42,21\}, \{168,84\}, \{45,127\}, \{180,90\}, \{93,151\}, \{163,186\}, \{25,117\}\\
	&&\{100,50\}, \{191,200\}, \{137,173\}, \{130,65\}, \{102,51\}, \{199,204\}, \{169,189\}\\
	&&\{49,129\}, \{196,98\}, \{157,183\}, \{3,106\}\\
	&\cup& \{12,6\}, \{48,24\}, \{192,96\}, \{141,175\}, \{146,73\}, \{166,83\}, \{37,123\},\\
	&& \{148,74\}, \{174,87\}, \{69,139\}, \{67,138\}, \{59,134\}, \{27,118\}, \{108,54\},\\ 
	&&\{14,7\}, \{56,28\}, \{15,112\}, \{60,30\}, \{31,120\}, \{124,62\}, \{78,39\},\\ 
	&&\{103,156\}, \{203,206\}, \{185,197\}, \{113,161\}, \{34,17\}, \{136,68\}, \{126,63\},\\ 
	&&\{86,43\}, \{135,172\}, \{122,61\}, \{70,35\}, \{71,140\}, \{75,142\}, \{91,150\},\\ 
	&&\{155,182\}, \{202,101\}, \{181,195\}, \{97,153\}, \{179,194\}, \{89,149\},\\
	&&\{147,178\}, \{170,85\}, \{53,131\}.
\end{eqnarray*}

\

{\bf Acknowledgment}

Research was partially supported by SNI and CONACyT.

\bibliographystyle{amsplain}

\end{document}